\documentclass{amsart}

\usepackage{amssymb}
\usepackage{amsfonts}
\usepackage{amsmath}
\usepackage{amsthm}
\usepackage{mathrsfs}
\usepackage{dsfont}
\usepackage{theoremref}
\usepackage[colorlinks=true, urlcolor=blue, citecolor=blue, linkcolor=blue]{hyperref}

\usepackage{aliascnt}

\newtheorem{thm}{Theorem}[section]

\newaliascnt{lem}{thm}  
\newtheorem{lem}[lem]{Lemma}
\aliascntresetthe{lem}

\newaliascnt{qst}{thm}  
\newtheorem{qst}[qst]{Question}
\aliascntresetthe{qst}

\newaliascnt{prp}{thm}  
\newtheorem{prp}[prp]{Proposition}
\aliascntresetthe{prp}

\newaliascnt{cor}{thm}  
\newtheorem{cor}[cor]{Corollary}
\aliascntresetthe{cor}

\theoremstyle{definition}
\newaliascnt{dfn}{thm}  
\newtheorem{dfn}[dfn]{Definition}
\aliascntresetthe{dfn}

\numberwithin{equation}{section}

\usepackage{a4wide}

\author{Tristan Bice}
\address{
York University\\
Toronto\\
Canada
}
\email{Tristan.Bice@gmail.com}
\thanks{This research has been supported by a York University Post-Doctoral Fellowship}


\begin{document}

\title{The Projection Calculus}

\begin{abstract}
We develop some tools for manipulating and constructing projections in C*-algebras.  These are then applied to give short proofs of some standard projection homotopy results, as well as strengthen some fundamental classical results for C*-algebras of real rank zero, specifically on liftings, excising pure states and Kadison's transitivity theorem.  Lastly, we investigate some order properties of the set of projections in C*-algebras of real rank zero, building on the work in \cite{Bice2009}.
\end{abstract}
\maketitle

\section{Introduction}

\subsection{Motivation}

As every operator algebraist knows, student and researcher alike, the continuous functional caluculus is an indispensable tool for doing anything beyond the bare basics of C*-algebra theory.  What has not been known until now (or at least not made explicit in the literature as far as we can tell) is that there is a projection analog that is equally powerful when applied to C*-algebras containing many projections, like those of real rank zero.  Our purpose in this paper is to develop this projection calculus and show how it can be applied to simplify and strengthen a number of previous results regarding projections in C*-algebras.

To motivate this, let us first back up and review a little C*-algebra history.  One of the first kinds of C*-algebras to be studied were those consisting of all continuous functions on a compact topological space $X$ with pointwise addition and multiplication together with the supremum norm, denoted by $C(X)$.  The celebrated Gelfand representation theorem tells us that in fact every unital commutative C*-algebra is (isometrically) isomorphic to one of this form.  While these commutative C*-algebras were originally studied for their own sake, it was soon realized that Gelfand's theorem yields a powerful functional calculus that can be applied to any normal operator even in a non-commutative C*-algebra.  Specifically, we take a normal operator $T\in A$ and note that the unital C*-subalgebra it generates is isomorphic to $C(\sigma(T))$, where $\sigma(T)$ is the spectrum of $T$.  Under this isomorphism, $T$ corresponds to the identity function $\mathrm{id}\in C(\sigma(T))$ and, for any $f\in C(\sigma(T))$, there is a corresponding operator $f(T)\in A$.  Any algebraic or norm relation that holds between $\mathrm{id}$ and $f$ in $C(\sigma(T))$ will of course still hold between $T$ and $f(T)$ in $A$, and so this calculus gives us a tool for constructing many new operators in $A$ having a desired relationship to the given $T$.

Fast forward a few decades and we start to see another elementary kind of C*-algebra being studied primarily for its own intrinsic interest, namely the C*-algebra generated by two projections $Q$ and $R$.  In \cite{Pedersen1968}, it is shown that such a C*-algebra is always isomorphic to a certain subalgebra of $C(\sigma(QR))\otimes M_2$ (which can be viewed as either $M_2(C(\sigma(QR)))$ or $C(\sigma(QR),M_2)$), and this is taken a step further in \cite{RaeburnSinclair1989}, where another proof of this result and some applications are also given.  We take this work to its natural conclusion, giving a general framework for these results to be applied in a similar manner to the classical continuous functional calculus.  For this identification of the C*-subalgebra of $A$ generated by $Q$ and $R$ allows us to construct many new operators in $A$, projections in particular, that have a desired relationship to the given $Q$ and $R$.  Specifically, we show how to construct, for any $f\in C(\sigma(PQ))$ with $\mathrm{ran}(f)\subseteq[0,1]$, $f(0)=0$ and $f(1)=1$ (if $1\in\sigma(QR)$), a projection $P=P_{Q,R,f}$ such that \[QPQ=f(QRQ).\]  Furthermore, $P$ will be Murray-von Neumann equivalent to $R$ and the construction will be continuous in $Q$, $R$ and $f$.

When put into its historical context in this way, the projection calculus seems like a very natural thing to develop, and the ideas underlying it are most likely already known at least intuitively by many researchers in operator algebras.  However, we believe this is the first time it has been made explicit, and the benefit of doing this is that it allows various generalizations and simplifications of previous results to be accomplished with relative ease (for example, see \S\ref{PH} below).

\subsection{Outline}

Before launching into the mathematics proper, we outline the structure of this article in a little more detail.  Firstly, in \S\ref{pre}, we mention and prove some basic facts that will be needed for the work that follows (primarily for \S\S\ref{PB} and \ref{ESKT}, so those wishing to see a quick derivation and application of the projection calculus should jump straight to \S\S\ref{PFC} and \ref{PH}).  Much of this material will be familiar, or at least intuitively obvious, to anyone with some knowledge of C*-algebras.  However, our approach using support projections and quasi-inverses is perhaps somewhat novel and allows for an expedient development of the necessary results (compare our simple derivation of the formula for the norm of an idempotent in (\ref{||id||}) with that in \cite{KolihaRakocevic2004}, for example).

In \S\ref{PFC} we do the aforementioned construction of a projection $P$ from a given pair of projections $Q$ and $R$, together with a function $f\in C(\sigma(QR))$ with $\mathrm{ran}(f)\subseteq[0,1]$, $f(0)=0$ and $f(1)=1$ (if $1\in\sigma(QR)$).  While it would be possible to do this using the description of the C*-subalgebra generated by $Q$ and $R$ as a subalgebra of $C(\sigma(QR))\otimes M_2$, as shown in \cite{AnoussisKatavolosTodorov2007}, \cite{Pedersen1968} or \cite{RaeburnSinclair1989}, we take a more elementary approach, using nothing more than the usual continuous functional calculus (showing that the projection calculus is not only similar to, but can also be derived from, the usual continuous functional calculus).  This has the advantage of making our exposition more elementary and self-contained, as well as easing some of the necessary calculations.  

As an immediate demonstration of the power of this new tool, in \S\ref{PH} we show how it can be used to simplify the proofs of two standard theorems about the existence of projection homotopies.  The projection calculus is then further applied in \S\ref{PB} to produce some strong lifting results for C*-algebras of real rank zero.  While also of independent interest, we suspect these lifting results will be useful in studying how certain properties of C*-algebras of real rank zero are preserved under homomorphisms.

It is in \S\ref{ESKT} that we give the most interesting application of the projection calculus we have developed thus far, using it to strengthen, in the real rank zero case, two fundamental results in C*-algebra theory.  The first of these says that pure states on C*-algebras of real rank zero can be excised exactly on projections.  This, in turn, allows us to prove a strong version of Kadison's transitivity theorem for C*-algebras of real rank zero, showing that irreducible representations are not just onto arbitrary finite dimensional subalgebras but also one-to-one when restricted to appropriate subalgebras.

In the final section we examine the canonical order on projections in C*-algebras of real rank zero, extending the work of \cite{Bice2009}.  It seemed appropriate to include these results in a paper on projections in C*-algebras of real rank zero, even though it does not require the projection calculus developed in \S\ref{PFC}.  It is, however, made easier with the notation and theory developed in \S\ref{pre}.

\section{Preliminaries}\label{pre}

\subsection{Well-Supported Operators and Quasi-Inverses}

The following definition is taken from \cite{Blackadar2006} II.3.2.8, and other equivalants can be found in \cite{Blackadar2006} II.3.2.11.  We denote the range and kernel of an operator $T$ by $\mathcal{R}(T)$ and $\mathcal{N}(T)$ respectively.
\begin{dfn}\label{wellsup}
We say an operator $T$ on a Hilbert space $H$ is \emph{well-supported} if any of the following equivalent conditions hold.
\begin{enumerate}
\item\label{ws1} $\inf(\sigma(TT^*)\backslash\{0\})>0$.
\item $\mathcal{R}(T)$ is closed.
\item $\inf_{v\in\mathcal{N}(T)^\perp,||v||=1}||Tv||>0$.
\end{enumerate}
\end{dfn}
One simple observation that will be used later is the following.  If $P$, $T$ and $S$ are operators such that $P$ is a projection and $PTS=P$ then $PT$ is well-supported, which follows from the fact $\mathcal{R}(P)$ is closed and $\mathcal{R}(P)=\mathcal{R}(PTS)\subseteq\mathcal{R}(PT)\subseteq\mathcal{R}(P)$.

We see that \ref{ws1} can be used as the definition of a well-supported element $T$ of an abstract C*-algebra $A$.  For well-supported $T\in A$, the characterstic function $\chi$ of the interval $(0,\infty)$ will be continuous on $\sigma(TT^*)$ and we can define $[T]=\chi(TT^*)\in A$.\footnote{This notation comes from \cite{Pedersen1979}, although there $[T]$ is used to denote the projection onto the closure of the range of an arbitrary element $T$ of a von Neumann algebra.  We use it in the more general context of C*-algebras but only for well-supported $T$.}  This $[T]$ is the \emph{left support projection} of $T$ and, with respect to any (faithful) representation of $A$, we have \[\mathcal{R}([T])=\mathcal{R}(T).\]   

Also, in what follows, we use the spectral family notation from \cite{Weidmann1980} so, for any self-adjoint operator $S$ on a Hilbert space $H$, $E_S(t)$ refers to the spectral projection of $S$ corresponding to the interval $(-\infty,t]$ and, likewise, $E_S(t-)$ refers to the spectral projection of $S$ corresponding to the interval $(-\infty,t)$.  We also write $E^\perp_S(t)$ for $E_S(t)^\perp=1-E_S(t)=$ the spectral projection of $T$ corresponding to the interval $(t,\infty)$.  The following proposition is a generalization of the easily verified fact that, for any $\lambda>0$, an operator $T$ on a Hilbert space takes $\lambda$-eigenvectors of $T^*T$ to $\lambda$-eigenvectors of $TT^*$.

\begin{prp}\label{TT*}
For any Hilbert space $H$, $T\in\mathcal{B}(H)$ and $t>0$, we have $E^\perp_{TT^*}(t)=[TE^\perp_{T^*T}(t)]$.
\end{prp}

\begin{proof} First note $\mathcal{R}(T^*TE_{T^*T}(t))\subseteq\mathcal{R}(E_{T^*T}(t))\perp\mathcal{R}(E^\perp_{T^*T}(t))$ so $\mathcal{R}(TE_{T^*T}(t))\perp\mathcal{R}(TE^\perp_{T^*T}(t))$.  Also, $\mathcal{R}(E_{TT^*}(0))=\mathcal{R}(T)^\perp$, so we just need to show $\langle TT^*Tv,Tv\rangle\leq t\langle Tv,Tv\rangle$, for $v\in\mathcal{R}(E_{T^*T}(t))$, and $\langle TT^*Tv,Tv\rangle> t\langle Tv,Tv\rangle$, for $v\in\mathcal{R}(E^\perp_{T^*T}(t))\backslash\{0\}$.  But this follows from the immediately verified fact that $T^*T(t-T^*T)E_{T^*T}(t)$ and $T^*T(T^*T-t)E^\perp_{T^*T}(t)$ are positive and strictly positive operators respectively.
\end{proof}

The following corollary provides us with a simple trick that will be very useful in manipulating operator expressions involving the continuous functional calculus.

\begin{cor}\label{g(TT^*)T}
For any C*-algebra $A$, $T\in A$ and continuous $g$ on $\mathbb{R}_+$, $Tg(T^*T)=g(TT^*)T$.
\end{cor}

\begin{proof}
Representing $A$ on a Hilbert space, we immediately see that $Tg(T^*T)$ and $g(TT^*)T$ both map $\mathcal{N}(T)=\mathcal{N}(T^*T)$ to $0$, while they also agree on $\mathcal{N}(T)^\perp$, by \autoref{TT*}.
\end{proof}

Let $f$ be the function on non-negative reals satisfying $f(0)=0$ and $f(t)=1/t$, for $t>0$.  Then $f$ is continuous on the spectrum of any well-supported positive operator and hence, for any such $S$ in a C*-algebra $A$, we have another operator $f(S)\in A$ which we will denote by $S^{-1}$.  As we shall see, this is quite a convenient convention, although we must be careful to keep in mind now that the notation $S^{-1}$ \emph{does not necessarily imply that $S$ is invertible}, only that it is well-supported (although $S^{-1}$ will indeed be the inverse when $S$ is invertible).  So in general we only have $SS^{-1}=[S]=[S^{-1}]=S^{-1}S$ (rather than the usual $SS^{-1}=1=S^{-1}S$).  We can even extend this to non-self-adjoint (but still well-supported) $T$ by defining \[T^{-1}=T^*(TT^*)^{-1}.\]  Then $TT^{-1}=[T]$ and, by applying \autoref{g(TT^*)T}, $T^{-1}T=[T^*]=[T^{-1}]$.  Also, $T^{-1}[T]=T^{-1}=[T^*]T^{-1}$ which, if $T=SP$ for some $S$ and projection $P$, means that $P(SP)^{-1}=(SP)^{-1}$ and $(PS)^{-1}P=(PS)^{-1}$.  Thus $TT^{-1}T=[T]T=T$ and $T^{-1}TT^{-1}=T^{-1}[T]=T^{-1}$, showing that $T^{-1}$ is the \emph{quasi-inverse} of $T$ in the ring-theoretic sense.  In fact, the well-supported elements of $A$ are precisely those with a quasi-inverse, as noted in \cite{Blackadar2006} II.3.2.10.\footnote{Other common names for the quasi-inverse are \emph{Moore-Pensrose inverse} (see \cite{Penrose1955}), \emph{generalized inverse} or \emph{pseudoinverse}, often denoted by $T^\dagger$ rather than $T^{-1}$ as we have done here.  There is a substantial body of literature on these (see \cite{Ben-IsraelGreville2003}), although mostly dealing with finite matrices, and sometimes von Neumann algebras, rather than the C*-algebras we consider here.}

If $TS=[T]$ then $T^{-1}TS=T^{-1}[T]$ and hence $[T^*]S=T^{-1}$.  On the other hand, if $T^{-1}=[T^*]S=T^{-1}TS$ then $[T]=TT^{-1}=TT^{-1}TS=[T]TS=TS$, i.e.
\begin{equation}\label{qinv}
TS=[T]\qquad\Leftrightarrow\qquad[T^*]S=T^{-1}.
\end{equation}
In particular, as $T^{-1}T=[T^*]=[T^{-1}]$, this means that $(T^{-1})^{-1}=[T^{-1*}]T=[T]T=T$.  Moreover, we can calculate the norm of $||T^{-1}||$ by first noting that
\[T^{-1*}T^{-1}=(TT^*)^{-1}TT^*(TT^*)^{-1}=(TT^*)^{-1}[TT^*]=(TT^*)^{-1}\textrm{ and hence, if $T\neq0$,}\]
\begin{equation}\label{||T^{-1}||}
||T^{-1}||^2=||(TT^*)^{-1}||=1/\min(\sigma(TT^*)\backslash\{0\}).
\end{equation}
As mentioned in \cite{Blackadar2006} II.3.2.9, we also have a polar decomposition for well-supported $T$.  Specifically, defining $|T|=\sqrt{T^*T}$ and $U=U_T=T|T|^{-1}$, we see that $T=U|T|$,
\begin{eqnarray*}
UU^* &=& T(T^*T)^{-1/2}(T^*T)^{-1/2}T^*=T(T^*T)^{-1}T^*=(TT^*)^{-1}TT^*=[TT^*]=[T],\\
U^*U &=& (T^*T)^{-1/2}T^*T(T^*T)^{-1/2}=[T^*T]=[T^*]\textrm{ and}\\
T^*U &=& T^*T(T^*T)^{-1/2}=|T|\in A_+
\end{eqnarray*}

\subsection{Projections I}

In what follows we will use the following elementary facts.  Firstly, for any C*-algebra $A$ and $S,T\in A$, $\sigma(ST)\backslash\{0\}=\sigma(TS)\backslash\{0\}$.  For $P,Q\in\mathcal{P}(A)$ (or even arbitrary idempotent $P,Q\in A$), we have $0\notin\sigma(PQ)\Leftrightarrow P=1=Q$ and hence \[\sigma(PQP)=\sigma(PPQ)=\sigma(PQ)=\sigma(QP)=\sigma(QQP)=\sigma(QPQ).\]  As $PQ^\perp P=P(1-PQP)P$, $\sigma(PQ)\cap(0,1)=1-\sigma(PQ^\perp)\cap(0,1)$ which, applied twice, gives \[\sigma(PQ)\cap(0,1)=\sigma(P^\perp Q^\perp)\cap(0,1).\]  Also, $(P-Q)=PQ^\perp-P^\perp Q$ so \[||P-Q||=\max(||PQ^\perp||,||P^\perp Q||).\]  In fact, as $\max(\sigma(PQ^\perp)\backslash\{1\})=\max(\sigma(P^\perp Q)\backslash\{1\})$, \[||P-Q||<1\quad\Rightarrow\quad||P-Q||=||PQ^\perp||=||P^\perp Q||.\]  Furthermore, \[||PQ^\perp||<1\quad\Leftrightarrow\quad PQ\textrm{ is well-supported and }[PQ]=P.\]  If, instead, $||P^\perp Q||<1$ then $PQ$ is again well-supported (because $QP$ is) although we may not have $[PQ]=P$ (but $[PQ]$ will be a continuous function in this case \textendash\, see \autoref{[]cont}).

Also note that if $P$ and $Q$ are projections with $||P^\perp Q||<1$ then $PQ$ is well-supported and
\begin{eqnarray*}
[PQ]Q^\perp[PQ] &=& [PQ]PQ^\perp P[PQ]=[PQ]-PQP[PQ]=(1-PQP)[PQP],\textrm{ hence}\\
||[PQ]Q^\perp||^2 &=& ||(1-PQP)[PQP]||=1-\min(\sigma(PQ)\backslash\{0\})=\max(\sigma(P^\perp Q)\backslash\{1\})=||P^\perp Q||^2.
\end{eqnarray*}
Furthermore, $[PQ]^\perp Q=Q-[PQ]Q=Q-PQ=P^\perp Q$, so
\begin{equation}\label{||Q-[PQ]||}
||Q-[PQ]||=\max(||[PQ]Q^\perp||,||[PQ]^\perp Q||)=||P^\perp Q||.
\end{equation}
Also, $(P-[PQ])Q=PQ-[PQ]PQ=PQ-PQ=0$, so $P-[PQ]+Q$ is a projection and \[||P-[PQ]+Q-P||=||Q-[PQ]||=||P^\perp Q||.\]  In particular, this observation simplifies and strengthens the Proposition in \cite{Blackadar2006} II.3.3.5.

\begin{prp}\label{PveeQ}
Take a C*-algebra $A$ and $P,Q\in\mathcal{P}(A)$ with $||PQ||<1$.  Then there exists $P\vee Q\in\mathcal{P}(A)$ such that, w.r.t. any representation, $\mathcal{R}(P\vee Q)=\mathcal{R}(P)+\mathcal{R}(Q)$.
\end{prp}

\begin{proof}
As $\sigma(PQ)\cap(0,1)=\sigma(P^\perp Q^\perp)\cap(0,1)$ and $\sqrt{||PQP||}=||PQ||<1$, we see that $1-P^\perp Q^\perp P^\perp$ is well-supported and we may define $P\vee Q=[1-P^\perp Q^\perp P^\perp]\in A$.  But \[\overline{\mathcal{R}(P)+\mathcal{R}(Q)}=(\mathcal{R}(P^\perp)\cap\mathcal{R}(Q^\perp))^\perp=\mathcal{R}(E_{P^\perp Q^\perp P^\perp}(1-))=\mathcal{R}(E^\perp_{1-P^\perp Q^\perp P^\perp}(0))=\mathcal{R}(P\vee Q).\]  Also, as $1-P^\perp Q^\perp P^\perp$ is a polynomial expression of $P$ and $Q$, we have $\mathcal{R}(P\vee Q)\subseteq\mathcal{R}(P)+\mathcal{R}(Q)$.
\end{proof}

\subsection{Idempotents}

Various facts relating idempotents\footnote{In Banach space theory, idempotent operators are sometimes called (oblique) projections, however for us the term projection always means \emph{orthogonal} projection, i.e. not just idempotent but also self-adjoint.} and projections in C*-algebras have been proved and reproved a number of times in the literature (see \cite{KolihaRakocevic2004} for an account of this history).  One classical result, from \cite{Greville1974} and \cite{Penrose1955}, says that idempotent matrices are just quasi-inverses of projection pair products.  This was generalized to operators on any Hilbert space in \cite{CorachMaestripieri2010} and we generalize this to arbitrary C*-algebras here (see also \cite{CorachMaestripieri2011}, for some other characterizations of projection products).

\begin{prp}\label{idin}
Assume $A$ is a C*-algebra.  Then $I\in A$ is idempotent if and only if there exist (necessarily unique) $P,Q\in\mathcal{P}(A)$ such that $||P-Q||<1$ and $I=(PQ)^{-1}$.
\end{prp}

\begin{proof}
If $P,Q\in A$ and $PQ$ is well-supported then $(PQ)^{-1}=Q(PQ)^{-1}=(PQ)^{-1}P$, and hence \[(PQ)^{-1}(PQ)^{-1}=(PQ)^{-1}PQ(PQ)^{-1}=(PQ)^{-1}[PQ]=(PQ)^{-1},\]
i.e. $(PQ)^{-1}$ is idempotent.  On the other hand, given idempotent $I\in A$ and assuming $A$ is represented on a Hilbert space, we see that $\mathcal{R}(I)=\mathcal{N}(1-I)$ is closed and hence $I$ is well-supported.  Thus $Q=[I]\in\mathcal{P}(A)$ and, likewise, $P=[I^*]\in\mathcal{P}(A)$.  We then have $QI=I=IP$ and $PI^*=I^*=I^*Q$.  As $I$ is idempotent, we also have $IQ=Q=QI^*$ and $I^*P=P=PI$.  Thus $PQI=PI=P$ and $QPI^*=QI^*=Q$, showing that $PQ$ and $QP$ are well-supported (see the observation after \autoref{wellsup}) with $[PQ]=P$ and $[QP]=Q$, and hence $||P-Q||<1$.  It also shows $(PQ)^{-1}=[QP]I=QI=I$ (see (\ref{qinv})).
\end{proof}

Say we have a C*-algebra $A$ and $P,Q\in\mathcal{P}(A)$ with $||PQ||<1$.  Then $Q^\perp P$ is well-supported and $(Q^\perp P)^{-1}$ is an idempotent.\footnote{Although we may have $||Q^\perp-P||=||Q^\perp P^\perp||=1$, in which case $Q^\perp$ will not be the projection appearing on the left in the formula in \autoref{idin}, that will actually be the smaller projection $(P\vee Q)Q^\perp$.}  Also, $(Q^\perp P)^{-1}=(Q^\perp P)^{-1}Q^\perp$, giving $(Q^\perp P)^{-1}Q=0$ and \[(Q^\perp P)^{-1}P=(Q^\perp P)^{-1}Q^\perp P=[(Q^\perp P)^*]=[PQ^\perp]=P.\]  Moreover, $(Q^\perp P)^*(P\vee Q)^\perp=PQ^\perp(P\vee Q)^\perp=P(P\vee Q)^\perp=0$ and hence $(Q^\perp P)^{-1}(P\vee Q)^\perp=0$.  Likewise, we see that $(P^\perp Q)^{-1}P=0$, $(P^\perp Q)^{-1}Q=Q$ and $(P^\perp Q)^{-1}(P\vee Q)^\perp=0$.  Thus $((P^\perp Q)^{-1}+(Q^\perp P)^{-1})P=P$, $((P^\perp Q)^{-1}+(Q^\perp P)^{-1})Q=Q$ and $((P^\perp Q)^{-1}+(Q^\perp P)^{-1})(P\vee Q)^\perp=0$, i.e.
\begin{equation}\label{idPveeQ}
(P^\perp Q)^{-1}+(Q^\perp P)^{-1}=P\vee Q.
\end{equation}
Moreover, $\min(\sigma(P^\perp Q)\backslash\{0\})=\min(\sigma(PQ^\perp)\backslash\{0\})=1-\max(\sigma(PQ))=1-||PQ||^2$ so, by (\ref{||T^{-1}||}),
\begin{equation}\label{||id||}
||(P^\perp Q)^{-1}||=||(Q^\perp P)^{-1}||=1/\sqrt{1-||PQ||^2},
\end{equation}
so long as $P,Q\neq0$.

\subsection{Projections II}

\begin{lem}\label{[]cont}
The function $[PQ]$ is continuous on $\{(P,Q)\in\mathcal{P}(\mathcal{B}(H))^2:||P^\perp Q||<1\}$.
\end{lem}

\begin{proof}
Taking $P,Q,R\in\mathcal{P}(\mathcal{B}(H))$ with $||P^\perp Q||,||P^\perp R||<1$, we have
\begin{eqnarray*}
[PR]^\perp[PQ] &=& [PR]^\perp PQ(PQ)^{-1}=[PR]^\perp P(Q-R)(PQ)^{-1}\quad\textrm{and hence}\\
||[PR]^\perp[PQ]|| &\leq& ||Q-R||||(PQ)^{-1}||=||Q-R||/\sqrt{1-||P^\perp Q||^2}.
\end{eqnarray*}
Likewise, $||[PQ]^\perp[PR]||\leq||Q-R||/\sqrt{1-||P^\perp R||^2}$ so \[||[PQ]-[PR]||\leq||Q-R||/\sqrt{1-\max(||P^\perp Q||,||P^\perp R||)^2}.\]  Similaraly, we see that, for $P,Q,R\in\mathcal{P}(\mathcal{B}(H))$ with $||P^\perp Q||,||R^\perp Q||<1$, \[||[PQ]-[RQ]||\leq||P-R||/\sqrt{1-\max(||P^\perp Q||,||R^\perp Q||)^2}.\]  Combine these inequalities to see that $[PQ]$ is continuous in both coordinates simultaneously.
\end{proof}

\begin{lem}\label{veecont}
The function $P\vee Q$ is continuous on $\{(P,Q)\in\mathcal{P}(\mathcal{B}(H))^2:||PQ||<1\}$.
\end{lem}

\begin{proof}
Take $P,Q,R\in\mathcal{P}(\mathcal{B}(H))$ with $||PQ||,||PR||<1$ and note that, by (\ref{idPveeQ}), \[(P\vee R)^\perp(P\vee Q)=(P\vee R)^\perp(Q^\perp P)^{-1}+(P\vee R)^\perp(P^\perp Q)^{-1}=(P\vee R)^\perp Q(P^\perp Q)^{-1}.\]  Thus $||(P\vee R)^\perp(P\vee Q)||\leq||(P\vee R)^\perp Q||||(P^\perp Q)^{-1}||\leq||R^\perp Q||/\sqrt{1-||PQ||^2}$.  Likewise, we have $||(P\vee Q)^\perp(P\vee R)||\leq||Q^\perp R||/\sqrt{1-||PR||^2}$ so \[||(P\vee Q)-(P\vee R)||\leq||Q-R||/\sqrt{1-\max(||PQ||,||PR||)^2}.\]  The function $P\vee Q$ is symmetric so the same equalities hold for the other coordinate and combining these shows that $P\vee Q$ is continuous in both coordinates.
\end{proof}

Finally, a few calculations.  Take a C*-algebra $A$ and $P,Q,R\in\mathcal{P}(A)$ with $R<P$.  For $\lambda\in[0,1]$,
\begin{equation}
||PQP-\lambda P||=\max(||PQ||^2-\lambda,\lambda-1+||PQ^\perp||^2).
\end{equation}
Also, $||RQR-\lambda R||\leq||PQP-\lambda P||$ and $||(P-R)Q(P-R)-\lambda(P-R)||\leq||PQP-\lambda P||$ so
$||\lambda R-RQR+\lambda(P-R)-(P-R)Q(P-R)||\leq||PQP-\lambda P||$.  Also 
\begin{eqnarray*}
||(P-R)QR|| &=& ||RQ(P-R)||=||RQ(P-R)+(P-R)QR||\quad\textrm{and}\\
RQ(P-R)+(P-R)QR &=& PQP-RQR-(P-R)Q(P-R)\\
&=& PQP-\lambda P+\lambda R-RQR+\lambda(P-R)-(P-R)Q(P-R),\textrm{ so}\\
||(P-R)QR|| &\leq& 2||PQP-\lambda P||.
\end{eqnarray*}
The optimal value of $\lambda$ is $(||PQ||^2+1-||PQ^\perp||^2)/2$, which gives
\begin{equation}\label{||(P-R)QR||}
||(P-R)QR||\leq||PQ||^2+||PQ^\perp||^2-1,\quad\textrm{and hence, if }||PQ^\perp||<1,
\end{equation}
\[||(P-R)[QR]||\leq||(P-R)QR||/\sqrt{1-||Q^\perp R||^2}\leq(||PQ||^2+||PQ^\perp||^2-1)/\sqrt{1-||PQ^\perp||^2}.\]  As $||[QR]R||=||QR||\leq||QP||$, if $||PQ||<1$ too then
\begin{eqnarray}
||(P-R)(R\vee[QR])|| &=& ||(P-R)[QR](R^\perp[QR])^{-1}||\nonumber\\
&\leq& (||PQ||^2+||PQ^\perp||^2-1)/\sqrt{(1-||PQ^\perp||^2)(1-||PQ||^2)}.\label{split}
\end{eqnarray}
In particular, if $||PQ||^2+||PQ^\perp||^2=1$, i.e. if $PQP=\lambda P$ for some $\lambda$, then $(P-R)(R\vee[QR])=0$.

\subsection{Partial Isometries}

\begin{prp}\label{UPQ}
Assume $U$ is a partial isometry, $P=U^*U$, $Q=UU^*$ and $U^*U^2$ is self-adjoint.  The following are equivalent.
\begin{enumerate}
\item\label{P-Q} $||P-Q||<1$.
\item\label{U-U*} $||U-U^*||<1$.
\item\label{U^2} $U^2$ is well-supported and $[U^2]=Q$
\end{enumerate}
\end{prp}

\begin{proof}
First note that $U^2=UU^*U^2=UU^{*2}U=QP$, so $||P^\perp Q||\leq||P-Q||<1$ implies $U^2$ is well-supported and $[U^2]=Q$, i.e. \ref{P-Q}$\Rightarrow$\ref{U^2}.  If $U^2$ is well-supported then so is $U^{*2}$, and if $[U^2]=Q$ then $[U^{*2}]=[U^{*2}U]=[U^*U^2]=[U^*Q]=P$.  As $U^{*2}=U^{*2}UU^*=U^*U^2U^*=PQ$, \ref{U^2} implies that $QP$ and $PQ$ are well-supported with $[PQ]=P$ and $[QP]=Q$, i.e. $||P-Q||<1$.  Also \[(U-U^*)(U-U^*)^*=(U-U^*)(U^*-U)=Q-QP-PQ+P=(P-Q)^2,\] and hence $||U-U^*||=||P-Q||$, which proves \ref{P-Q}$\Leftrightarrow$\ref{U-U*}.
\end{proof}

It is well known that any pair of projections $P$ and $Q$ in a C*-algebra $A$ satisfying $||P-Q||<1$ are Murray-von Neumann equivalent, as witnessed by the partial isometry $U_{QP}$ coming from the polar decomposition of $QP$.  More precisely, this yields a one-to-one correspondence between projection pairs $P,Q\in A$ such that $||P-Q||<1$, and partial isometries $U\in A$ such that $U^*U^2$ is positive and $||U-U^*||<1$.\footnote{or, equivalently for such partial isometries, $||UU^*-U||<\sqrt{2}$ or even $||1-U||<\sqrt{2}$.}

\begin{prp}\label{U}
Let $A$ be a C*-algebra.  For $P,Q\in\mathcal{P}(A)$ with $||P-Q||<1$, $U_{QP}$ is the unique partial isometry $U$ such that $U^*U=P$, $UU^*=Q$ and $U^*U^2\in A_+$.
\end{prp}

\begin{proof}
If $U=U_{QP}$ then $UU^*=[QP]=Q$, $U^*U=[PQ]=P$ and $U^*U^2=PQU=|QP|\in A_+$.  On the other hand, say $U$ is another partial isometry with $U^*U=P$, $UU^*=Q$ and $U^*U^2\in A_+$, which also means $U^2U^*=U(U^*U^2)U^*\in A_+$.  By \autoref{UPQ}, $P=[U^{*2}]=[U^{*2}U]=[U^*P]$.  Also, $PQP=U^*UUU^*U^*U=(U^*U^*U)^2=(U^*P)^2$ and hence, \[U_{QP}=QP(PQP)^{-1/2}=UU^*P(U^*P)^{-1}=U[U^*P]=UP=U.\]
\end{proof}

The partial isometries above form a subclass of the collection of split partial isometries considered in \cite{AndruchowCorachMbekhta2011}.  However, $U^*U^2\in A_+$ alone does not imply that $U$ is a split partial isometry.  In fact, at the other extreme we can have (non-zero) partial isometries $U$ with $U^2=0$.  However, if $U^*U^2$ is positive and well-supported, then we can always split up $U$ into a partial isometry of this form plus a partial isometry of the form in \autoref{U}.

In fact, say $U$ is a partial isometry and $P\leq U^*U$ is a projection commuting with $U^*U^2$.  Then we may let $Q=UPU^*$, $P'=U^*U-P$ and $Q'=UU^*-Q=UP'U^*$.  As \[U^*UQ=U^*U^2PU^*=PU^*U^2PU^*=PUPU^*=PQ,\] so $P'Q=(U^*U-P)Q=0$.  Likewise, $PUU^*=PU^*U^2U^*=PU^*U^2PU^*=PQ$ and hence $PQ'=P(UU^*-Q)=0$.  Thus if $U^*U^2$ is self-adjoint and well-supported, letting $P_+=[(U^*U^2)_+]$, $P_-=[(U^*U^2)_-]$, $P_0=U^*U-P_+-P_-$, $U_+=UP_+$, $U_-=-UP_-$ and $U_0=UP_0$, we see that $U=U_+-U_-+U_0$ and, when represented on a Hilbert space, $\mathcal{R}(U_k)+\mathcal{R}(U_k^*)$ are mutually othogonal subspaces for $k=+,-,0$.

\section{The Projection Calculus}\label{PFC}

The general situation we want to consider is as follows.  We are given two projections $Q$ and $R$ in a C*-algebra $A$, together with a continuous function $f$ from $\sigma(QR)$ to $[0,1]$ with $f(0)=0$ and $f(1)=1$ (if $1\in\sigma(QR)$).  We want to obtain another projection $P=P_{Q,R,f}$ in $A$ onto a subspace obtained, roughly speaking, by moving the eigenvectors of $RQR$ in the range of $R$ towards or away from $Q$ so that $\lambda$-eigenvectors of $RQR$ become $f(\lambda)$-eigenvectors of $PQP$.  Equivalently, we want $\lambda$-eigenvectors of $QRQ$ to be $f(\lambda)$-eigenvectors of $QPQ$ which, stated more precisely in the language of the continuous functional calculus, means $QPQ=f(QRQ)$.  To obtain this $P$ we apply the continuous functional calculus to $RQR$ in the following way.

To begin with, we will further assume that, if $0$ is a limit point of $\sigma(QR)$, $f(s)/s$ has a limit as $s$ approaches $0$ in $\sigma(QR)\backslash\{0\}$ and, if $1$ is a limit point of $\sigma(QR)$, $(1-f(s))/(1-s)$ has a limit as $s$ approaches $1$ in $\sigma(QR)\backslash\{1\}$, i.e. $f$ has a (finite) derivative at $0$ and $1$.  This ensures that there are continuous functions $x_f$ and $y_f$ on $\sigma(QR)$ with $x_f(s)=\sqrt{f(s)/s}$ for $s\neq0$ and $y_f(s)=\sqrt{(1-f(s))/(1-s)}$ for $s\neq1$.  We can then define $U=U_{Q,R,f}$ by
\begin{equation}
U=QRx_f(RQR)+Q^\perp Ry_f(RQR)
\end{equation}
Note that $R(RQR)=RQR=(RQR)R$ so $R$ commutes with $y_f(RQR)$ and, as $f(0)=0$, $Rf(RQR)=f(RQR)$ so
\begin{eqnarray*}
U^*U &=& x_f(RQR)RQRx_f(RQR)+y_f(RQR)RQ^\perp Ry_f(RQR)\\
&=& f(RQR)+Ry_f(RQR)(1-RQR)y_f(RQR)\\
&=& f(RQR)+R(1-f(RQR))\\
&=& R,
\end{eqnarray*}
i.e. $U$ is a partial isometry with initial projection $R$.  We define $P=P_{Q,R,f}$ to be the final projection of $U$, i.e. $P=UU^*$.  Applying \autoref{g(TT^*)T} with $T=QR$, we get
\begin{equation}\label{QPQ}
QPQ=QRx_f(RQR)^2RQ=x_f(QRQ)^2QRQ=f(QRQ),
\end{equation}
as required.

Even with very simple functions $f$, the projection calculus is surprisingly powerful.  In fact, for our applications in the following sections $f$ will always be piecewise linear and, in particular, have a derivative at $0$ and $1$.  However, for the sake of interest and completeness, we now show that this restriction is not essential.  Specifically, note first that
\begin{eqnarray*}
U_{Q,R,f}-U_{Q,R,g} &=& QR(x_f-x_g)(RQR)+Q^\perp R(y_f-y_g)(RQR)\quad\textrm{so}\\
(U_{Q,R,f}-U_{Q,R,g})(U_{Q,R,f}-U_{Q,R,g})^* &=& (x_f-x_g)(RQR)RQR(x_f-x_g)(RQR)\\
&& +R(y_f-y_g)(RQR)(1-RQR)(y_f-y_g)(RQR)\\
&=& a_{f,g}(RQR)\quad\textrm{and hence}\\
||U_{Q,R,f}-U_{Q,R,g}|| &=& \sup_{s\in\sigma(QR)}\sqrt{a_{f,g}(s)},\quad\textrm{where}\\
a_{f,g}(s) &=& (\sqrt{f(s)}-\sqrt{g(s)})^2+(\sqrt{1-f(s)}-\sqrt{1-g(s)})^2\\
&=& 2(1-\sqrt{f(s)g(s)}-\sqrt{(1-f(s))(1-g(s))}).
\end{eqnarray*}
Thus, even if $f$ does not have a derivative at $0$ and $1$, we can still define a sequence $(f_n)$ of functions which does, such that $f_n$ approaches $f$ uniformly.  Even though the functions $x_{f_n}$ might increase to infinity at some points, the working above shows that the partial isometries $U_{Q,R,f_n}$ will still converge, necessarily to another partial isometry $U_{Q,R,f}$ whose final projection $P=P_{Q,R,f}$ satisfies $QPQ=f(QRQ)$.

We can also obtain a nice formula expressing the norm difference of the resulting projections $P_f=P_{Q,R,f}$ and $P_g=P_{Q,R,g}$.  First note that
\begin{eqnarray*}
U^*_fU_g &=& (x_f(RQR)RQ+y_f(RQR)RQ^\perp)(QRx_g(RQR)+Q^\perp Ry_g(RQR))\\
&=& x_f(RQR)RQRx_g(RQR)+y_f(RQR)RQ^\perp Ry_g(RQR)\\
&=& b_{f,g}(RQR)R,\quad\textrm{where}\\
b_{f,g}(s) &=& \sqrt{f(s)g(s)}+\sqrt{(1-f(s))(1-g(s))}.
\end{eqnarray*}
Also note that, for any $s,t\in[0,1]$,
\begin{eqnarray*}
1-(\sqrt{st}+\sqrt{(1-s)(1-t)})^2 &=& 1-(st+2\sqrt{(st)(1-s)(1-t)}+(1-s)(1-t))\\
&=& s+t-2(\sqrt{(st(1-s)(1-t))}+st)\\
&=& (1-s)t+s(1-t)-2\sqrt{st(1-s)(1-t))}\\
&=& (\sqrt{(1-s)t}-\sqrt{s(1-t)})^2.
\end{eqnarray*}
Thus
\[R-U^*_gU_fU^*_fU_g=R(1-b_{f,g}^2)(RQR)=Rc_{f,g}^2(RQR)=c_{f,g}^2(RQR),\]
where $c_{f,g}(s)=\sqrt{(1-f(s))g(s)}-\sqrt{f(s)(1-g(s))}$ (note that $c_{f,g}(0)=0$).  So
\begin{eqnarray*}
||P_f^\perp P_g||^2 &=& ||P_gP_f^\perp P_g||=||U_gU^*_gP_f^\perp U_gU^*_g||=||U^*_gP_f^\perp U_g||=||R-U^*_gU_fU^*_fU_g||\\
&=& \max_{s\in\sigma(QR)}c^2_{f,g}(s).
\end{eqnarray*}
However, $U^*_fU_g=b_{f,g}(RQR)R$ is self-adjoint so $R-U^*_gU_fU^*_fU_g=R-U^*_fU_gU^*_gU_f$ and hence
\begin{eqnarray*}
||P_f-P_g|| &=& ||P_f^\perp P_g||=||P_g^\perp P_f||=\max_{s\in\sigma(QR)}|c_{f,g}(s)|\\
&=& \max_{s\in\sigma(QR)}|\sqrt{(1-f(s))g(s)}-\sqrt{f(s)(1-g(s))}|.
\end{eqnarray*}
Also note that when $g=\mathrm{id}$ we have $U_g=P_g=R$ so, with $P=P_f$, the above formula becomes
\begin{equation}\label{||P-R||}
||P-R||=\max_{s\in\sigma(QR)}|\sqrt{(1-f(s))s}-\sqrt{f(s)(1-s)}|.
\end{equation}

Lastly, we point out that the equation $QPQ=f(QRQ)$ will usually not have a unique solution, even for projections $P$ in the C*-subalgebra generated by the given $Q$ and $R$.  However, we think that the $P_f$ we have constructed is in some sense the most natural choice.  For example, we would conjecture that, given any $Q_f,Q_g\in\mathcal{P}(A)$ with $QQ_fQ=f(QRQ)$ and $QQ_gQ=g(QRQ)$, we necessarily have $||Q_f-Q_g||\geq||P_f-P_g||$ and probably also $P_fQ_gP_f\leq P_fP_gP_f$.

\section{Projection Homotopies}\label{PH}

The previous section concludes the development of the projection calculus, which we will see applied to prove a number of new results in the following sections.  First, however, we give a couple of quick examples to see how it can be used to simplify proofs of classical results, like the following.

\begin{thm}\label{QRhomo}
Any projections $Q$ and $R$ in a C*-algebra $A$ with $||Q-R||<1$ are homotopic.
\end{thm}

By now, this is a standard result, although proving it usually requires a little bit of effort (see \cite{Wegge-Olsen1993} Proposition 5.2.6, for example).  But with the projection calculus in our toolbox, it becomes almost trivial.  Simply take any homotopy from the identity $\mathrm{id}$ on $\sigma(QR)$ to the characteristic function $\chi$ of $\sigma(QR)\backslash\{0\}$, e.g. set $f_t(x)=(1-t)x+t$ for $x\in\sigma(QR)\backslash\{0\}$ and $f_t(0)=0$, for all $t\in[0,1]$.  Then note that $P_t=P_{Q,R,f_t}$ is the required homotopy of projections from $R$ to $Q$.  Furthermore, the projection calculus gives us more control over the kind of homotopy used so, if our application required the homotopy to avoid some finite collection of projections, for example, we could easily achieve this by simply adjusting the homotopy from $\mathrm{id}$ to $\chi$ accordingly.

In a similar vein, we have the following, which is \cite{Wegge-Olsen1993} Proposition 5.3.8.

\begin{thm}
If $Q$ and $R$ are Murray-von Neumann equivalent projections in a C*-algebra $A$ satisfying $||QR||<1$ then they are necessarily homotopic.
\end{thm}

The proof in \cite{Wegge-Olsen1993} goes by first proving the case $QR=0$, using \cite{Wegge-Olsen1993} Proposition 4.2.7, and we give essentially the same argument here.  Specifically, let $U$ be the partial isometry with $U^*U=Q$ and $UU^*=R$ and note that $S=U+U^*$ is self-adjoint and $S^2$ is the projection $R+Q$.  Thus $S=R'-Q'$ for some $R',Q'\in\mathcal{P}(A)$ with $R'+Q'=R+Q$ and $R'Q'=0$.  Let $f$ move $1$ to $-1$ within the unit circle of $\mathbb{C}$, e.g. let $f(t)=e^{t\pi i}$ for $t\in[0,1]$.  Setting $S_t=R'-f(t)Q'$, we have $S_0RS_0=U^*RU=Q$, $S_1RS_1=(R+Q)R(R+Q)=R$ and $S_tRS_t\in\mathcal{P}(A)$, for all $t\in[0,1]$, and so $S_tRS_t$ is the required homotopy.

The rest of the proof in \cite{Wegge-Olsen1993} goes by reducing the general case to the $QR=0$ case with a couple of pages of spectral theoretical tricks, as the author calls them.  However, we can do the same reduction in just a couple of lines by using the projection calculus.  Simply take any homotopy from the identity to the zero function on $\sigma(QR)$, e.g. set $f_t(x)=(1-t)x$, for all $t\in[0,1]$.  Then $P_t=P_{Q,R,f_t}$ is a homotopy from $R$ to a projection $P$ satisfying $QPQ=f_1(QRQ)=0$, which is Murray-von Neumann equivalent to $R$, and hence to $Q$, by construction.

\section{Lifting}\label{PB}

Say we have a homomorphism $\pi$ from a C*-algebra $A$ onto another C*-algebra $B$.  In this section, we consider the problem of lifting an operator $t\in B$ to an operator $T\in A$ (i.e. satisfying $\pi(T)=t$) with the same properties.  For example, we might want to lift self-adjoint operators in $B$ to self-adjoint operators in $A$ (Loring \cite{Loring1997} would say rather that we are lifting the relation $T=T^*$), or likewise with projections, idempotents or partial isometries instead.  We might also require something extra, like that the norm or spectrum of the lifting remains the same as the original (the norm and spectrum of the lifting can not possibly be smaller but they can certainly be much larger).  In fact, we might actually want to lift many operators simultaneously and ensure that some relationships between them remain the same.  For example, say we want to show that the collection of projections in the Calkin algebra $\mathcal{C}(H)$ has no $(\omega,\omega)$-gaps.  The collection of projections in $\mathcal{B}(H)$ certainly has no $(\omega,\omega)$-gaps (the supremum of the bottom half or the infimum of the top half will always interpolate any pregap of projections in a von Neumann algebra), so we will be done so long as we can lift $(\omega,\omega)$-gaps in $\mathcal{P}(\mathcal{C}(H))$ to $(\omega,\omega)$-gaps in $\mathcal{P}(\mathcal{B}(H))$.  This can be done recursively, so long as we can lift any $p,q\in\mathcal{P}(\mathcal{C}(H))$ with $p\leq q$ to $P,Q\in\mathcal{P}(\mathcal{B}(H))$ with $P\leq Q$ (actually what is required here is slightly stronger, namely what is called a two-step lifting in \cite{Loring1997} Definition 8.1.6 \textendash\, see below or \cite{Bice2009} \S3 for more details).

Going back to the first lifting problem above, we see that it is always possible to lift a self-adjoint operator $s\in B$ to a self-adjoint operator $S\in A$ \textendash\, simply take any $T\in A$ with $\pi(T)=s$ and let $S=\frac{1}{2}(T+T^*)$ (see \cite{Loring1997} for a whole host of other relations that can be lifted in general C*-algebras).  We can even ensure that the norm remains the same, i.e. $||S||=||s||$, using the usual continuous functional calculus.  But with projections, we are already in trouble.  For we could have $A=C([0,1])$, $B=\mathbb{C}\oplus\mathbb{C}$ and $\pi(f)=(f(x),f(y))$, for all $f$ and some distinct $x,y\in[0,1]$.  Then $A$ has only two projections, the unit and zero, neither of which map to the projection $(0,1)\in B$.  Similarly, the other problems are not solvable in general, and so we have to restrict the class of C*-algebras under consideration if we want to obtain general solutions.  It turns out that a nice general class in which many of these problems have quite general solutions is the class of C*-algebras of real rank zero.  These have a number of different characterizations (see \cite{Blackadar2006} V 3.2.9, for example), although the most important for our work involves the existence of spectral projection approximations (see \cite{Brown1991} and \cite{BrownPedersen1991}), as given below.  As will be seen in what follows, this is an extremely useful (and, up till now, widely underutilized) characterization of real rank zero C*-algebras.
\begin{dfn}\label{rr0}
A C*-algebra $A$ has real rank zero if, for all $s>t>0$ and self-adjoint $S\in A$, there exists $P\in\mathcal{P}(A)$ such that $E^\perp_S(s)\leq P\leq E^\perp_S(t)$.
\end{dfn}
This definition uses spectral projections and so it might appear to be dependent on the particular Hilbert space $H$ we are considering $A$ to be represented on.  However, this is not the case and we could, for example, state $E^\perp_S(s)\leq P$ more precisely in the abstract C*-algebra context as $f_n(S)\leq P$, for all $n$, where $f_n:\mathbb{R}\rightarrow[0,1]$ is a sequence of continuous functions increasing pointwise to the characteristic function of $(s,\infty)$.

Now it follows fairly directly from \autoref{rr0} that if $A$ (and hence $B$ too) has real rank zero then projections can indeed always be lifted to projections (for another proof in the case when $\pi$ is the canonical map to the Calkin algebra see \cite{Weaver2007}, for example), and that we can even preserve the ordering, as required for the $(\omega,\omega)$-gap problem mentioned above.  More precisely, for any $p\in\mathcal{P}(B)$ and $Q\in\mathcal{P}(A)$ with $p\leq\pi(Q)$ we can choose $P\in\mathcal{P}(A)$ so that we not only have $\pi(P)=p$ but also $P\leq Q$ (see \cite{Bice2009} Theorem 3.4).  Note that, for projections $P$ and $Q$, $P\leq Q$ is equivalent to $||Q^\perp P||=0$.  So we could ask, more generally, if, given $p\in\mathcal{P}(B)$ and $Q\in\mathcal{P}(A)$, we necessarily have $P\in\mathcal{P}(A)$ with \[\pi(P)=p\qquad\textrm{and}\qquad||PQ||=||\pi(PQ)||.\]  As just mentioned, the $||\pi(Q)p||=0$ case has been proved (in \cite{Bice2009}), while the $||\pi(Q)p||=1$ case is trivial.  So assume $||\pi(Q)p||=\sqrt{\lambda}\in(0,1)$ and take $R\in\mathcal{P}(A)$ with $\pi(R)=p$.  For any $\epsilon\in(0,1-\lambda)$ we can find $P\in\mathcal{P}(A)$ with $E^\perp_{RQ^\perp R}(1-\lambda-\epsilon/2-)\leq P\leq E^\perp_{RQ^\perp R}(1-\lambda-\epsilon)$.  Then \[p=E^\perp_{\pi(RQ^\perp R)}(1-\lambda-\epsilon/2-)\leq\pi(P)\leq E^\perp_{\pi(RQ^\perp R)}(1-\lambda-\epsilon)=p\quad\textrm{and}\quad||PQ||\leq\sqrt{\lambda+\epsilon}.\]  So we can at least get arbitrarily close to our goal.  To reach it, we use the projection calculus.

\begin{thm}\label{||PQ||}
Assume $\pi$ is a C*-algebra homomorphism from $A$ to $B$ and we have $R,Q\in\mathcal{P}(A)$ with $||QR||<1$.  Then there exists $P\in\mathcal{P}(A)$ with $\pi(P)=\pi(R)$ and $||PQ||=||\pi(PQ)||$.
\end{thm}

\begin{proof}
Let $f$ be the function on $\sigma(QR)$ with $f(s)=s$, for $s\leq||\pi(QR)||^2$, and $f(s)=||\pi(QR)||^2$, for $s\geq||\pi(QR)||^2$.  Set $P=P_{Q,R,f}$, and note that $\pi(P)=P_{\pi(Q),\pi(P),f}$.  For all $s\in\sigma(\pi(QR))\subseteq[0,||\pi(QR)||^2]$, $f(s)=s$ and hence $||\pi(P)-\pi(R)||=0$, by (\ref{||P-R||}), i.e. $\pi(P)=\pi(R)$.  But we also have $||PQ||^2=||QPQ||=||f(QRQ)||=||\pi(QR)||^2=||\pi(QP)||^2$.
\end{proof}

Thus if $\pi$ is a homomorphism from a C*-algebra $A$ of real rank zero onto $B$ and we have $p\in\mathcal{P}(B)$ and $Q\in\mathcal{P}(A)$ then we can indeed find $P\in\mathcal{P}(A)$ with $\pi(P)=p$ and $||PQ||=||\pi(PQ)||$.  But note that $||PQ||^2=||PQP||=\max(\sigma(PQP))=\max(\sigma(PQ))$ and so the following theorem shows we can do even better in the real rank zero case.  In fact, as $\sigma(PQ)$ (almost) completely determines how $P$ and $Q$ would be spatially related when represented on a Hilbert space, the following theorem means that any pair of projections can be lifted to another pair with the same spatial relationship, giving the strongest lifting result for a pair of projections that we could possibly hope for.

\begin{thm}\label{sigmapq}
Assume $\pi$ is a homomorphism from a C*-algebra $A$ of real rank zero onto $B$.  For any $p\in\mathcal{P}(B)\backslash\{1\}$ and $Q\in\mathcal{P}(A)$, we have $P\in\mathcal{P}(A)$ with $\pi(P)=p$ and $\sigma(PQ)=\sigma(\pi(PQ))$.
\end{thm}

\begin{proof}
As $\sigma(p\pi(Q))$ is closed, there exists a sequence of disjoint open intervals $(I_n)\subseteq(0,1)$ such that $(0,1)\backslash\sigma(p\pi(Q))=\bigcup_nI_n$.  For each $n$, let $s_n,t_n\in(0,1)$ be such that $I_n=(s_n,t_n)$, set $r_n=(s_n+t_n)/2$ and define $f_n$ and $g_n$ on $[0,1]$ by
\begin{eqnarray*}
f_n(r) &=& r,\textrm{ for }r\leq s_n,\\
f_n(r) &=& s_n,\textrm{ for }r\geq s_n,\\
g_n(r) &=& t_n,\textrm{ for }r\leq t_n,\textrm{ and}\\
g_n(r) &=& r,\textrm{ for }r\geq t_n.
\end{eqnarray*}
Recursively define $(P_n)\subseteq\mathcal{P}(A)$ as follows.  Let $P_0$ be any projection in $A$ with $\pi(P_0)=p$ and, given $n$, take $\delta>0$ and $E_-,E,E_+\in\mathcal{P}(A)$ such that \[E^\perp_{P_nQP_n}(r_n+2\delta)\leq E_+\leq E^\perp_{P_nQP_n}(r_n+\delta)\leq E\leq E^\perp_{P_nQP_n}(r_n-\delta)\leq E_-\leq E^\perp_{P_nQP_n}(r_n-2\delta).\]  Let $P=E_--E_+$ and $R=E-E_+$.  By choosing $\delta$ sufficiently small, we can ensure that $S=[(R\vee[QR])^\perp(P-R)]$ is well defined and $||S-(P-R)||$ is as small as we like, by (\ref{split}).

We claim that $SE=0=SQE$ or, equivalently, assuming $A$ is represented on a Hilbert space, $\mathcal{R}(S)\perp\mathcal{R}(E)+\mathcal{R}(QE)$.  To see this, note that, as $P-R\leq E_{P_nQP_n}(r_n+\delta)$ and $P-R\leq P_n$, \[\mathcal{R}(P-R)\perp V=\mathcal{R}(E^\perp_{P_nQP_n}(r_n+\delta))+\mathcal{R}(QE^\perp_{P_nQP_n}(r_n+\delta)).\]  Letting $R'=E-E^\perp_{P_nQP_n}(r_n+\delta)$, by the same reasoning we have $\mathcal{R}(R')\perp V$ and hence also $\mathcal{R}(QR')\perp V$.  Thus $\mathcal{R}((R'\vee[QR'])^\perp(P-R))\perp V$.  As $R-R'\leq E^\perp_{P_nQP_n}(r_n+\delta)$, it follows that $(R'\vee[QR'])^\perp(P-R)=(R\vee[QR])^\perp(P-R)$ and hence $\mathcal{R}(S)\perp V$.  As $E-R\leq E^\perp_{P_nQP_n}(r_n+\delta)$, we have $\mathcal{R}(E)+\mathcal{R}(QE)=V+\mathcal{R}(R)+\mathcal{R}(QR)$.  We certainly have $\mathcal{R}(S)\perp\mathcal{R}(R)+\mathcal{R}(QR)$ and hence, finally, $\mathcal{R}(S)\perp\mathcal{R}(E)+\mathcal{R}(QE)$.

From this it follows that, setting $T=S\vee(P_n-E_-)$, we have $TE=TQE=0$.  Thus $P_{Q,T,f_n}P_{Q,E,g_n}=0$ and we may define the projection $P_{n+1}$ to be $P_{Q,T,f_n}+P_{Q,E,g_n}$, completing the recursion.  As in \autoref{||PQ||}, it follows that $\pi(P_n)=p$, for all $n\in\omega$.  From (\ref{||P-R||}) it follows that $||P_{Q,E,g_n}-E||,||P_{Q,T,f_n}-T||\leq\sqrt{(1-s_n)t_n}-\sqrt{s_n(1-t_n)}$.  We can also ensure that, at each stage of the recursion, $||S-(P-R)||$ is small enough that $||T-(P_n-E)||<\sqrt{(1-s_n)t_n}-\sqrt{s_n(1-t_n)}$ so $||P_{n+1}-P_n||\leq2\sqrt{(1-s_n)t_n}-\sqrt{s_n(1-t_n)}$.  In fact, at each stage of the recursion, we only modify $P_n$ to obtain $P_{n+1}$ on the ($P_n$ and $Q$ invariant) subspace $\mathcal{R}(E')+\mathcal{R}(QE')$, where $E'=E_{P_0QP_0}(t_n)-E_{P_0QP_0}(s_n)$.  These subspaces are perpendicular for distinct $n$ and hence, for $m>n$, \[||P_m-P_n||\leq\max_{k\geq n}2\sqrt{(1-s_k)t_k}-\sqrt{s_k(1-t_k)}.\]  The function $\sqrt{(1-s)t}-\sqrt{s(1-t)}$ is continuous on $[0,1]\times[0,1]$ and $0$ on the diagonal, and hence approaches $0$ as $|s-t|\rightarrow0$.  This means that $(P_n)$ is a Cauchy sequence and has a limit $P\in\mathcal{P}(A)$.  As $\pi(P_n)=p$, for all $n$, we certainly have $\pi(P)=p$.  For $n>m$, the projection $F$ onto $\mathcal{R}(E^\perp_{P_{m+1}QP_{m+1}}(r_m))+\mathcal{R}(QE^\perp_{P_{m+1}QP_{m+1}}(r_m))$ commutes with both $P_n$ and $Q$, and we have $||Q^\perp P_nF||^2\leq1-t_n^2$ and $||QP_nF^\perp||^2\leq s_n^2$.  Thus $F$ commutes with $P$ too and $||Q^\perp PF||^2\leq1-t_n^2$ and $||QPF^\perp||^2\leq s_n^2$, which means that $\sigma(PQ)\cap(s_n,t_n)=\emptyset$.  So $\sigma(PQ)\cap(0,1)=\sigma(\pi(PQ))\cap(0,1)$ and, as $p\neq1$, $0\in\sigma(PQ)\cap\sigma(\pi(PQ))$.  If $1\notin\sigma(\pi(PQ))$ then $PQ^\perp$ must be well-supported and we may simply replace $P$ with $[PQ^\perp]$ to obtain $\sigma(PQ)=\sigma(\pi(PQ))$.
\end{proof}

With $P$ as above, it automatically follows that we also have $\sigma(PQ^\perp)\backslash\{1\}=\sigma(\pi(PQ^\perp))\backslash\{1\}$.  If $1\notin\sigma(\pi(PQ^\perp))$ then, as in the last line of the proof, $PQ$ is well-supported and so we may replace $P$ with $[PQ]$ to actually obtain $\sigma(PQ^\perp)=\sigma(\pi(PQ^\perp))$.

\begin{cor}
Assume $\pi$ is a homomorphism from a C*-algebra $A$ of real rank zero onto $B$.  For any idempotent $i\in\mathcal{P}(B)\backslash\{1\}$, we have idempotent $I\in A$ with $\pi(I)=i$ and $\sigma(I^*I)=\sigma(i^*i)$.
\end{cor}

\begin{proof}
By \autoref{idin}, we have $p,q\in\mathcal{P}(B)$ with $i=(pq)^{-1}$.  By \autoref{sigmapq}, we have $P,Q\in\mathcal{P}(A)$ such that $\pi(P)=p$, $\pi(Q)=q$ and $\sigma(PQ)=\sigma(pq)$.  Then $i^*i=(pqp)^{-1}$ so $\sigma(i^*i)=\{0\}\cup(\sigma(pq))^{-1}=\{0\}\cup(\sigma(PQ))^{-1}=\sigma(I^*I)$.
\end{proof}

\begin{thm}\label{pi(U)}
Assume $\pi$ is a homomorphism from a C*-algebra $A$ of real rank zero onto $B$.  For any partial isometry $u\in B$, we have a partial isometry $U\in A$ with $\pi(U)=u$ and $||U^2||=||u^2||$.
\end{thm}

\begin{proof}
Take $T\in A$ with $\pi(T)=u$ and let $P\in\mathcal{P}(A)$ be such that $E^\perp_{T^*T}(2/3)\leq P\leq E^\perp_{T^*T}(1/3)$, so $\pi(P)=u^*u$ and $\pi(TP)=uu^*u=u$.  Then $TP$ is well-supported and hence we have a partial isometry $U_{TP}\in A$ with $\pi(U)=u(\sqrt{u^*u})^{-1}=uu^*u=u$.  So $P=U_{TP}^*U_{TP}$ and we may let $Q=U_{TP}U_{TP}^*$.  Now let $R\in\mathcal{P}(A)$ be such that $\pi(R)=uu^*=\pi(Q)$ and $||PR||=||u^*u^2u^*||=||u^2||$.  By replacing $R$ with $R'$ such that $E^\perp_{RQR}(2/3)\leq R'\leq E^\perp_{RQR}(1/3)$ if necessary, we may assume that $||Q^\perp R||<1$.  Thus $RQ$ is well-supported, $[RQ]=R$ and, setting $U=U_{RQ}U_{TP}$, we have
\[UU^*=U_{RQ}U_{TP}U_{TP}^*U^*_{RQ}=U_{RQ}QU^*_{RQ}=U_{RQ}U^*_{RQ}=[RQ]=R.\]
Also $U^*U\leq P$ so $||U^2||=||U^*U^2U^*||\leq||PR||=||u^2||$ and $\pi(U)=U_{\pi(RQ)}u=uu^*u=u$.
\end{proof}

\begin{cor}
Assume $\pi$ is a homomorphism from a C*-algebra $A$ of real rank zero onto $B$.  For any partial isometry $u\in B\backslash\{1\}$ such that $u^*u^2$ is well-supported and positive, we have a partial isometry $U\in A$ with $\pi(U)=u$ and $\sigma(U)=\sigma(u)$.
\end{cor}

\begin{proof}
Split $u$ up into two partial isometries $u_0$ and $u_+$, where $u_0^2=0$ and $[u_+^2]=u_+u_+^*$, as mentioned after \autoref{U}.  By \autoref{pi(U)}, we have a partial isometry $U_0\in A$ such that $U_0^2=0$ and $\pi(U_0)=U_0$.  Take $P,Q\in\mathcal{P}(A)$ such that $\pi(P)=p=u_+^*u_+$, $\pi(Q)=q=u_+u_+^*$ and $P(U_0U_0^*+U_0^*U_0)=Q(U_0U_0^*+U_0^*U_0)=0$.  By the proof of \autoref{||PQ||}, we may also assume that $\sigma(PQ)=\sigma(pq)$.  Setting $U=U_0+U_{QP}$ we then have $\pi(U)=\pi(u)$ and \[\sigma(U)=\sigma(UU^*U)=\sigma(U^*U^2)=\sigma(|QP|)=\sqrt{\sigma(PQ)}=\sqrt{\sigma(pq)}=\sigma(u).\]
\end{proof}

It would be interesting to know if this corollary can be generalized, in particular if it holds for the case when $u^*u^2$ is only assumed to be self-adjoint and/or not necessarily well-supported.  In trying to extend this to the self-adjoint case we were lead to the following simple question.  Given $\pi$, $A$ and $B$ as above and $p,q,r\in\mathcal{P}(B)$ with $pqr=0=pr$, is it possible to find $P,Q,R\in\mathcal{P}(A)$ such that $\pi(P)=p$, $\pi(Q)=q$, $\pi(R)=r$ and $PQR=0=PR$?  For example, if $||pq||<1$ and $||pq^\perp||<1$ then the answer is yes, for we can take $S\in\mathcal{P}(A)$ with $\pi(S)=p\vee[qp]$ and then choose $P,Q_p\leq S$ and $R,Q_r\leq S^\perp$ such that $\pi(P)=p$, $\pi(Q_p)=[qp]$, $\pi(R)=r$ and $\pi(Q_r)=q-[qp]$.  Setting $Q=Q_p+Q_r$ then completes the set of required liftings.  If $B$ were the Calkin algebra, for example, then using the theory from \cite{Bice2009} and the fact that $B$ is $\sigma$-closed and has no $(\omega,\omega)$-gaps, we could, just under the assumption that $||pq^\perp||<1$, find $s\in\mathcal{P}(A)$ commuting with $q$ such that $p\leq s$ and $r\leq s^\perp$, and then perform the same argument with $S\in\mathcal{P}(A)$ such that $\pi(S)=s$ (and this would be enough to extend the above result to the case when $u^*u^2$ self-adjoint but still well-supported).  We do not know if the result holds in general, however.

\section{Excising Pure States and Kadison's Transitivity Theorem}\label{ESKT}

We now apply the theory developed so far to strengthen two fundamental C*-algebra results in the real rank zero case.  The first of these says that pure states on C*-algebras of real rank zero can be excised exactly on projections.  The original theorem it extends is the following version of \cite{AkemannAndersonPedersen1986} Proposition 2.2, which by now has become an essential tool in the operator algebraists toolbox (for example, see \cite{BrownOzawa2008} Lemma 1.4.11 for Kishimoto's slick proof of Glimm's lemma using it).

\begin{thm}
For any $\epsilon>0$, pure state $\phi$ on a C*-algebra $A$ and $S\in A^1_+$, there exists $T\in A^1_+$ such that $\phi(T)=1$ and $||TST-\phi(S)T^2||<\epsilon$.
\end{thm}

Note that this shows that pure states $\phi$ are completely determined by the positive norm $1$ elements on which they take the value $1$, as given by the formula $\phi(S)=\inf_{T\in A^1_+,\phi(T)=1}||TST||$, for $S\in A_+$.  For the Calkin algebra (or any other real rank zero C*-algebra for that matter \textendash\, see the the first paragraph of the proof below) $\phi$ will even be determined by just the projections on which it takes the value $1$, by the same formula $\phi(S)=\inf_{P\in\mathcal{P}(A),\phi(P)=1}||PSP||$, for $S\in A_+$.  In \cite{Bice2011}, we used this fact to obtain some new partial solutions to the long-standing Kadison-Singer conjecture, and then began to wonder if this $\inf$ was actually a $\min$, and more generally if we could eliminate the $\epsilon$ from the theorem above.  As pointed out to us by Mikael R\o rdam\footnote{who made another interesting comment regarding a modification of \autoref{ex}.  Specifically, if we eliminate the pure state and start with a non-central projection $Q$ then, for any $t\in[0,1]$, we can find a non-zero projection $P$ with $PQP=tP$, so long as our C*-algebra $A$ has property (SP) (which is significantly weaker than real rank zero, requiring only that all hereditary C*-subalgebras of $A$ contain a non-zero projection).}, this certainly can not be done, even in the real rank zero case, for arbitrary $S\in A^1_+$, as witnessed by the identity function in $A=C(X)$, where $X$ is the Cantor subset of $[0,1]$ (and $\phi$ is any pure state).  Originally, we felt that another counterexample for some projection $S$ could also be found in something like the Calkin algebra (we couldn't hope to find such an easy projection counterexample in a commutative C*-algebra as the $\epsilon$ can be trivially eliminated in this case by letting $T$ be $S$ or $S^\perp$).  Consequently, it came as a bit of a surprise to find the $\epsilon$ can be eliminated when $S$ is a projection, not just in the Calkin algebra but more generally in any C*-algebra of real rank zero, as we now show using the projection calculus.

Note that in this section, for $t\in[0,1]$, we define $P_{Q,R,t}=P_{Q,R,t\chi}$ where $\chi$ is the characteristic function of $(0,1]$ (or, strictly speaking, its restriction to $\sigma(PQ)$).

\begin{thm}\label{ex}
If $\phi$ is a pure state on a C*-algebra $A$ of real rank zero and $Q\in\mathcal{P}(A)$ then there exists $P\in\mathcal{P}(A)$ such that $\phi(P)=1$ and $PQP=\phi(Q)P$.
\end{thm}

\begin{proof}
By \cite{AkemannAndersonPedersen1986} Proposition 2.2, for any $\epsilon>0$ we have $R\in A^1_+$ such that $||RQR-\phi(Q)R||\leq\epsilon$ and $\phi(R)=1$.  As $A$ has real rank zero, we have $R'\in\mathcal{P}(A)$ such that $E^\perp_R(1-\epsilon)\leq R'\leq E^\perp_R(1-2\epsilon)$.  Thus $v_\phi\in\mathcal{R}(E^\perp_{\pi_\phi(R)}(1-\epsilon))\subseteq\mathcal{R}(\pi_\phi(R'))$ and hence $\phi(R')=1$.  Also $R'\leq E^\perp_R(1-2\epsilon)$ so $||(1-R)R'||\leq 2\epsilon$ and therefore $||R'QR'-\phi(Q)R'||\leq7\epsilon$ (alternatively note that, as $A$ has real rank zero, its hereditary subalgebras each contain an approximate unit of projections so, by the proof of \cite{AkemannAndersonPedersen1986} Proposition 2.2, we can actually choose $R\in\mathcal{P}(A)$ from the beginning).

So we may take $(R_n)\subseteq\mathcal{P}(A)$ such that $||R_nQR_n-\phi(Q)R_n||\rightarrow0$ and $\phi(R_n)=1$, for all $n\in\mathbb{N}$.  Furthermore, taking a positive sequence $(\epsilon_n)$ with $\epsilon_n\rightarrow0$ and replacing $R_{n+1}$ with $R'$ such that $E^\perp_{R_nR_{n+1}R_n}(1-\epsilon_n)\leq R'\leq E^\perp_{R_nR_{n+1}R_n}(1-2\epsilon_n)$, for all $n\in\mathbb{N}$, we may assume that $(R_n)$ is decreasing.  Finally, by replacing $R_n$ with $P_{Q,R_n,\phi(Q)}$, for each $n\in\mathbb{N}$, we instead have $R_nQR_n=\phi(Q)R_n$, for all $n\in\mathbb{N}$, $\phi(R_n)\rightarrow1$ and $||R_n^\perp R_{n+1}||\rightarrow0$.

By taking a subsequence if necessary, we may ensure that $||R^\perp_nR_{n+1}||<r_n$ for any positive sequence $(r_n)$ with $r_n\rightarrow0$.  For all $n\in\mathbb{N}$, let $P_n=R_n+P_{Q,S_{n-1},\phi(Q)}$ (with $S_{-1}=0$) and $S_n=[(R_{n+1}\vee[QR_{n+1}])^\perp(P_n-[P_nR_{n+1}])]$.  Thus, for all $n\in\mathbb{N}$, $S_{n-1}(R_n\vee[QR_n])=0$ and hence $(S_{n-1}\vee[QS_{n-1}])(R_n\vee[QR_n])=0$.  This means that $P_{Q,S_{n-1},\phi(Q)}(R_n\vee[QR_n])=0$ so $P_n$ is indeed a projection and $P_nQP_n=\phi(Q)P_n$, for all $n\in\mathbb{N}$.

By making $r_n$ is sufficiently small, we can make $||[P_nR_{n+1}]-R_{n+1}||=||P^\perp_nR_{n+1}||\leq||R^\perp_nR_{n+1}||$ as small as we like, by \autoref{[]cont}.  It follows we can make $||[QP_nR_{n+1}]-[P_nR_{n+1}]||$ and hence $||R_{n+1}\vee[QR_{n+1}]-[P_nR_{n+1}]\vee[QP_nR_{n+1}]||$ as small as we like too, by \autoref{veecont}.  As we have $([P_nR_{n+1}]\vee[QP_nR_{n+1}])(P_n-[P_nR_{n+1}])=0$, we can make $||S_n-(P_n-[P_nR_{n+1}])||$ as small as we like, again by \autoref{[]cont}.  As $P_nQP_n=\phi(Q)P_n$, we can therefore make $S_n-P_{Q,S_n,\phi(Q)}$ as small as we like.  As \[||P_n-P_{n+1}||\leq||[P_nR_{n+1}]-R_{n+1}||+||S_n-(P_n-[P_nR_{n+1}])||+||S_n-P_{Q,S_n,\phi(Q)}||,\] we can therefore make $||P_n-P_{n+1}||$ as small as we like.  Specifically, let us choose $(r_n)$ so that $||P_n-P_{n+1}||\leq2^n$, for all $n\in\mathbb{N}$.  This ensures that $(P_n)$ is Cauchy and has a limit $P\in\mathcal{P}(A)$.  As $P_nQP_n=\phi(Q)P_n$, for all $n\in\mathbb{N}$, $PQP=\phi(Q)P$, while $\phi(R_n)\leq\phi(P_n)\rightarrow1$ gives $\phi(P)=1$.
\end{proof}

We can now use this to strengthen Kadison's transitivity theorem in the real rank zero case.  There are a number of different ways of phrasing the original Kadison transitivity theorem, although the one we choose is the following.

\begin{thm}\label{KT}
If $\pi$ is an irreducible representation of a C*-algebra $A$ on a Hilbert space $H$ and $K$ is a finite dimensional subspace of $H$ then the map $T\mapsto\pi(T)|_K$ is onto $\mathcal{B}(K)$.
\end{thm}

In fact, using the self-adjoint version of the Kadison transitivity theorem (see \cite{Pedersen1979} Theorem 2.7.5), one can even strengthen this to the following.\footnote{Mikael R\o rdam has told us this idea goes back to Glimm.}

\begin{thm}\label{KTS}
If $\pi$ is an irreducible representation of a C*-algebra $A$ on a Hilbert space $H$ and $K$ is a finite dimensional subspace of $H$ then there exists a subalgebra $B$ of $A$ on which the map $T\mapsto\pi(T)|_K$ is a homomorphism onto $\mathcal{B}(K)$.
\end{thm}

The following corollary of \autoref{ex} shows that we can strengthen this further to an isomorphism so long as $A$ has real rank zero.  An important point to note is that the subalgebra $B$ below may not be unital.  Indeed, as pointed out to us by Ilijas Farah, $A$ could be the CAR algebra (i.e. the $2^\infty$-UHF C*-algebra) which is simple and hence any irreducible representation must be on an infinite dimensional Hilbert space.  Thus $K$ could have dimension $3$, meaning $B\cong M_3$, even though the only unital finite (full) matrix subalgebras of $A$ are isomorphic to $M_{2^n}$, for some $n$.

\begin{cor}\label{pi}
If $\pi$ is an irreducible representation of a C*-algebra $A$ of real rank zero on a Hilbert space $H$ and $K$ is a finite dimensional subspace of $H$ then there exists a subalgebra $B$ of $A$ on which the map $T\mapsto\pi(T)|_K$ is an isomorphism onto $\mathcal{B}(K)$.
\end{cor}

\begin{proof}
Let $e_1,\ldots,e_n$ be an orthonormal basis for $K$.  For $m=1,\ldots,n-1$, define a pure state $\phi_m$ on $A$ by $\phi_m(T)=\langle \pi(T)f_m,f_m\rangle$, where $f_m=\frac{1}{\sqrt{2}}e_m+\frac{1}{\sqrt{2}}e_{m+1}$.  By Kadison's transitivity theorem, we have $I\in\mathcal{P}(A)$ with $K\subseteq\mathcal{R}(\pi(I))$.  Likewise, we have $S\in A^1_+$ with $\pi(S)e_1=e_1$ and $\pi(S)e_2=\ldots=\pi(S)e_n=0$.  As $A$ has real rank zero, we therefore have $Q_1\in\mathcal{P}(A)$ such that $E^\perp_{ISI}(2/3)\leq Q_1\leq E^\perp_{ISI}(1/3)$, so $Q_1\leq I$, $\pi(Q_1)e_1=e_1$ and $\pi(Q_1)e_2=\ldots=\pi(Q_1)e_n=0$.

Recursively define projections $P_1,\ldots,P_{n-1}$ and $Q_2,\ldots,Q_n$ in $A$ as follows.  Once $Q_m$ has been defined, let $S\in A^1_+$ satsify $\pi(S)e_m=e_m$, $\pi(S)e_{m+1}=e_{m+1}$ and $\pi(S)e_{m+2}=\ldots=\pi(S)e_n=0$.  Take $R\in\mathcal{P}(A)$ such that $E^\perp_{T^*T}(2/3)\leq R\leq E^\perp_{T^*T}(1/3)$, where $T=S(I-(Q_1+\ldots+Q_{m-1}))$, so $R\leq I$, $RQ_1=\ldots=RQ_{m-1}=0$, $\pi(R)e_{m+2}=\ldots=\pi(R)e_n=0$, $\pi(R)e_m=e_m$ and $\pi(R)e_{m+1}=e_{m+1}$.  Thus we may take $R_0=R$ in the proof of \autoref{ex} to get $P_m\in\mathcal{P}(A)$ such that $P_mQ_mP_m=\phi_m(Q_m)P_m=\frac{1}{2}P_m$, as well as $P_mQ_1=\ldots=P_mQ_{m-1}=0$ and $\pi(P_m)e_{m+2}=\ldots=\pi(P_m)e_n=0$.  Set $Q_{m+1}=[Q^\perp_mP_m](=2Q^\perp_mP_mQ^\perp_m)$ and continue the recursion until $Q_n$ is defined.

Let $U_n=Q_n$ and, for $m=1,\ldots,n-1$, let $U_m=2Q_mP_mU_{m+1}$, so $U_m$ is a partial isometry with $U^*_mU_m=Q_n$ and $U_mU^*_m\leq Q_m$.  Our construction ensures that $\pi(U_m)e_l=\delta_{l,n}e_m$ and $\pi(U^*_m)e_l=\delta_{l,m}e_n$, for $l,m=1,\ldots,n$.  Thus the map $T\mapsto\pi(T)|_K$ is an isomorphism on the algebra $B$ generated by $U_1,\ldots,U_n$.
\end{proof}

In fact, the above theorem can even be generalized to finite collections of irreducible representations, as shown below.

\begin{cor}\label{finpi}
If $\pi_1,\ldots,\pi_n$ are inequivalent irreducible representations of a C*-algebra $A$ of real rank zero on Hilbert spaces $H_1,\ldots,H_n$ with finite dimensional subspaces $K_1,\ldots,K_n$ respectively, then there exists a subalgebra $B$ of $A$ on which the map $T\mapsto\pi_1(T)|_{K_1}\oplus\ldots\oplus\pi_n(T)|_{K_n}$ is an isomorphism onto $\mathcal{B}(K_1)\oplus\ldots\oplus\mathcal{B}(K_n)$.
\end{cor}

\begin{proof}
By Kadison's transitivity theorem, we have $I\in\mathcal{P}(A)$ with $K_m\subseteq\mathcal{R}(\pi_m(I))$, for $m=1,\ldots,n$, as well as $J_m\in\mathcal{P}(A)$ satisfying $K_m\subseteq\mathcal{R}(\pi_m(J_m))$ and $K_l\subseteq\mathcal{N}(\pi_m(J_m))$, for $m=1,\ldots n$ and $l=m+1,\ldots n$.  For $m=1,\ldots n$, let $I_m\in\mathcal{P}(A)$ satisfy $E^\perp_{T^*T}(2/3)\leq I_m\leq E^\perp_{T^*T}(1/3)$, where $T=J_m(I-(I_1+\ldots+I_{m-1}))$.  So $I_1,\ldots,I_n$ are pairwise orthogonal and $K_m\subseteq\mathcal{R}(\pi_m(I_m))$, for $m=1,\ldots,n$.  Thus we may proceed as in the proof of \autoref{pi} for each representation $\pi_m$, starting with $I_m$ in place of $I$.
\end{proof}

Irreducible representations on commutative algebras can be seen as points on the topological space defining the algebra, and hence \autoref{finpi} in the commutative case follows from the elementary fact that, given finitely many points in a zero dimensional Hausdorff space, there exist disjoint clopen subsets each containing precisely one of these points.  In fact, this completely characterizes zero dimensional spaces (as long as $X$ is locally compact), and so the following question naturally arises.

\begin{qst}
Does \autoref{finpi} completely characterize C*-algebras of real rank zero?
\end{qst}

\section{The Order On Projections in C*-Algebras of Real Rank Zero}

In this last section we continue some of the work done in \cite{Bice2009}, investigating order properties of the set of projections in C*-algebras of real rank zero, in particular looking at classical partial order concepts and examining their relation to certain quantum analogs.

First, recall that a partially ordered set is \emph{atomless} if every element has a strictly smaller lower bound.  For example, if $A$ is the Calkin algebra $\mathcal{C}(H)=\mathcal{B}(H)/\mathcal{K}(H)(=$ the collection of bounded operators on $H$ modulo the compact operators $\mathcal{K}(H))$ of any (infinite dimensional) Hilbert space $H$ then $\mathcal{P}(A)\backslash\{0\}$ is atomless, because every infinte dimensional subspace of $H$ contains another infinite dimensional subspace of infinite codimension.  The following lemma shows that if $A$ is a C*-algebra of real rank zero and $\mathcal{P}(A)\backslash\{0\}$ is atomless then a stronger quantum analog actually holds.
\begin{thm}\label{atomless}
Assume $A$ is a C*-algebra of real rank zero and $\mathcal{P}(A)\backslash\{0\}$ is atomless.  Then, for all $P,Q,R\in\mathcal{P}(A)\backslash\{0\}$, there exists $P',R'\in\mathcal{P}(A)\backslash\{0\}$ such that $P'\leq P$, $R'\leq R$ and $P'QR'=0$.
\end{thm}

\begin{proof}
Without loss of generality, we may assume that $||P-Q||<1$.  For if $PQ=0$, we are done, while if $||PQ||=\delta>0$ then, as $A$ has real rank zero, we may replace $P$ with $P'\in\mathcal{P}(A)$ such that $P'\geq E^\perp_{PQP}(\delta/2)$.  Then we may replace $Q$ with $[QP]$, because $Q(P-[QP])=0$ (see the comments before \autoref{PveeQ}).  Likewise, we may assume $||Q-R||<1$ (actually, we only need $||Q^\perp R||<1$).  As $\mathcal{P}(A)$ is atomless, we have $R'\in\mathcal{P}(A)\backslash\{0\}$ with $R'<R$.  Then $0<[QR']<Q$ and we may set $P'=[(QP)^{-1}(Q-[QR'])]$ so $0<P'<P$ and $P'QR'=0$.
\end{proof}

Next, recall that a partially ordered set is \emph{(downwards) $\sigma$-closed} if every decreasing sequence has a lower bound.  Again, if $A$ is the Calkin algebra, then $\mathcal{P}(A)\backslash\{0\}$ is $\sigma$-closed.  Indeed any decreasing sequence $(p_n)\subseteq\mathcal{P}(A)\backslash\{0\}$ can be lifted to decreasing $(P_n)\subseteq\mathcal{P}(\mathcal{B}(H))$.  We can then recursively construct a orthonormal sequence $(v_n)\subseteq H$ such that $v_n\in\mathcal{R}(P_n)$, for each $n$.  Letting $P$ be the projection onto $\overline{\mathrm{span}}(v_n)$, we see that $\pi(P)$ is a non-zero lower bound of $(p_n)$.  We also see that, again, if $A$ is a C*-algebra of real rank zero and $\mathcal{P}(A)\backslash\{0\}$ is $\sigma$-closed, then another stronger quantum analog actually holds.

\begin{lem}
If $A$ is a C*-algebra and $P_-,P_+,Q,R\in\mathcal{P}(A)$ satisfy $P_+P_-=0$, $R\leq P_++P_-$ and $||QP_-||<||QP_+||$ then
\begin{equation}\label{bigeq}
||P^\perp_+R||^2=||P_-R||^2\leq\frac{||QP_+||^2+||Q^\perp R||^2+||P_+Q^\perp P_-||-1}{||QP_+||^2-||QP_-||^2}
\end{equation}
\end{lem}

\begin{proof}
Assume $A$ is faithfully represented on a Hilbert space $H$.  For each unit vector $v\in\mathcal{R}(R)$, letting $v_+=P_+v$ and $v_-=P_-v$ we have
\begin{eqnarray*}
||Q^\perp R||^2 &\geq& ||Q^\perp v||^2\\
&=& ||Q^\perp v_+||^2+||Q^\perp v_-||^2-2\Re\langle Q^\perp v_+,v_-\rangle\\
&\geq& (1-||QP_+||^2)(1-||v_-||^2)+(1-||QP_-||^2)||v_-||^2-||P_+Q^\perp P_-||.
\end{eqnarray*}
Thus $(||QP_+||^2-||QP_-||^2)||v_-||^2\leq||QP_+||^2+||Q^\perp R||^2+||P_+Q^\perp P_-||-1$, from which (\ref{bigeq}) immediately follows.
\end{proof}

\begin{thm}\label{oprop}
Assume $A$ is a C*-algebra of real rank zero, $\mathcal{P}(A)\backslash\{0\}$ is $\sigma$-closed and $Q\in\mathcal{P}(A)$.  Then any decreasing $(P_n)\subseteq\mathcal{P}(A)\backslash\{0\}$ has a lower bound $P\in\mathcal{P}(A)\backslash\{0\}$ with $PQP=\lambda P$, where $\lambda=\inf||P_nQ||^2$.
\end{thm}

\begin{proof}
First note that if $\lambda=0$ the theorem is immediate, for then any lower bound $P$ of $(P_n)$ will satisfy $PQP=\lambda P$.  So we may assume $\lambda>0$.  Also, for any $(\epsilon_n)\subseteq\mathbb{R}_+$ decreasing to $0$, we may replace $(P_n)$ with a subsequence so that $||P_nQ||^2\leq\lambda+\epsilon_n$.  Furthermore, for any $(s_n)\subseteq\mathbb{R}_+$ increasing to $\lambda$ and $(\delta_n)\subseteq\mathbb{R}_+$ decreasing to $0$ (with $\delta_n\in(0,s_n)$, for each $n$), we have $(R_n)\subseteq\mathcal{P}(A)$ such that $E^\perp_{P_nQP_n}(s_n+\delta_n)\leq R_n\leq E^\perp_{P_nQP_n}(s_n-\delta_n)$, for all $n$.  From (\ref{||(P-R)QR||}) and (\ref{bigeq}), it follows that \[||R_n^\perp R_{n+1}||\leq\frac{\lambda+\epsilon_n-(s_{n+1}-\delta_{n+1})+2\delta_n}{\lambda-(s_n+\delta_n)}.\]  From this inequality it should be clear that, by choosing $(\epsilon_n)$, $(s_n)$ and $\delta_n$ appropriately, we can make $||R_n^\perp R_{n+1}||$ as small as we like, say, less than $1/2^n$.  In particular, this will ensure that, when we define $(T_n)\subseteq A$ recursively by $T_{n+1}=T_nR_{n+1}$ (and $T_1=R_1$), that $T_n$ is well-supported, for all $n$.  Thus $([T_n])\subseteq\mathcal{P}(A)\backslash\{0\}$ is decreasing and thus has a lower bound $R\in\mathcal{P}(A)$.  For all $n$, set $S_n=[T_n^{-1}R]$ and note that $S_n=[R_nS_{n+1}]$ and hence $||S_{n+1}-S_n||=||R_n^\perp S_{n+1}||\leq||R_n^\perp R_{n+1}||$ (see (\ref{||Q-[PQ]||})).  Thus $(S_n)$ is a Cauchy sequence and has a limit $P\in\mathcal{P}(A)\backslash\{0\}$.  For all $n<m$ we see that $||P^\perp_nP||\leq||R^\perp_m P||\rightarrow0$, so $P$ is indeed a lower bound of $(P_n)$.  But, as $||R_nQR_n-\lambda R_n||\rightarrow0$, we also have $PQP=\lambda P$.
\end{proof}

In particular, for any C*-algebra $A$ of real rank zero, $\mathcal{P}(A)\backslash\{0\}$ will be $\sigma$-closed if and only if $A$ has the \emph{(downwards) $\omega$-property}, as given in \cite{Bice2009} Definition 3.9,\footnote{Actually, this definition says that, given $Q\in\mathcal{P}(A)$ and decreasing $(P_n)\subseteq\mathcal{P}(A)$, if every lower bound of $(P_n)$ is a lower bound of $Q$ then we necessarily have $||Q^\perp P_n||\rightarrow0$.  So this version is equivalent to the version given here with $Q^\perp$ in place of $Q$.  The theorem and proof of \autoref{oprop} hold for $Q^\perp$ in place of $Q$ too and, in any case, the two versions are equivalent when $A$ is a unital C*-algebra.} i.e. given $Q\in\mathcal{P}(A)$, any decreasing $(P_n)\subseteq\mathcal{P}(A)$ with $\inf||P_nQ||>0$ has a lower bound $P\in\mathcal{P}(A)$ with $||PQ||>0$ (the forwards implication follows from \autoref{oprop}, while the reverse implication is immediate, even without the real rank zero assumption).  This $\omega$-property was used in \cite{Bice2009} to prove that the Calkin algebra contains no non-trivial countable gaps, even non-linear or non-commutative ones.  In the particular case of the Calkin algebra, it can be easily verified directly (i.e. without recourse to \autoref{oprop}) that $A$ has the $\omega$-property, in essentially the same way as $\mathcal{P}(A)\backslash\{0\}$ is shown to be $\sigma$-closed (see \cite{Bice2009} Theorem 3.10).  However, we began to wonder if these results could be generalized, and \autoref{oprop} shows that indeed they apply more generally to real rank zero C*-algebras $A$ with $\mathcal{P}(A)\backslash\{0\}$ $\sigma$-closed.\footnote{To be honest, though, the term `more generally' is perhaps not justified, as we do not know of any such C*-algebras that can not be obtained in some elementary way from the Calkin algebra itself.}

If $A$ is a unital C*-algebra then $\mathcal{P}(A)\backslash\{0\}$ will be downwards $\sigma$-closed if and only if $\mathcal{P}(A)\backslash\{1\}$ is upwards $\sigma$-closed, because the map $P\mapsto P^\perp$ is order inverting and takes $0$ to $1$.  On the other hand, what would naturally be considered as the \emph{upwards $\omega$-property}, i.e. given $Q\in\mathcal{P}(A)$, any increasing $(P_n)\subseteq\mathcal{P}(A)$ with $\sup||P_nQ||<1$ has an upper bound $P\in\mathcal{P}(A)$ with $||PQ||<1$, appears to be a fundamentally different property.  For one thing, all von Neumann algebras $A$ are immediately seen to have the upwards $\omega$-property, while they can only satisfy the downwards $\omega$-property vacuously.\footnote{For if a von Neumann algebra $A$ contains a strictly decreasing sequence $(P_n)$ of projections then this sequnce has a greatest lower bound $P$ and hence $(P_n-P)$ will be a decreasing sequence with no non-zero lower bound, i.e. $\mathcal{P}(A)\backslash\{0\}$ will not even be $\sigma$-closed.}  It also seems natural to conjecture that the upwards $\omega$-property is preserved under homomorphisms of C*-algebras of real rank zero, even though this is certainly not the case with the downwards $\omega$-property (see the discussion in \cite{Bice2009} after Definition 3.9).  One way of conceivably proving this would be to first strengthen \autoref{||PQ||}, i.e. to show that when $\pi$ is a C*-algebra homomorphism from $A$ to $B$ and $S,R,Q\in\mathcal{P}(A)$ with $||QR||<1$, $||QS||\leq||\pi(QR)||$ and $\pi(S)\leq\pi(R)$, there exists $P\in\mathcal{P}(A)$ with $\pi(P)=\pi(R)$, $||PQ||=||\pi(PQ)||$ and $S\leq P$.  However, we do not know if this holds or, indeed, if the Calkin algebra even has the upwards $\omega$-property.

The best we can do is show that this holds for the \emph{upwards $\lambda$-$\omega$-property}, for all $\lambda\in[0,1]$, i.e. given $Q\in\mathcal{P}(A)$, any increasing $(P_n)\subseteq\mathcal{P}(A)$ with $P_nQP_n=\lambda P_n$, for all $n$, has an upper bound $P\in\mathcal{P}(A)$ with $PQP=\lambda P$.\footnote{Incidentally, many C*-algebras that one encounters satisfy this property, although one that does not can be found in \cite{Akemann1970} Example I.2.}  For assume $A$ is a C*-algebra of real rank zero, $\pi$ is a homomorphism from $A$ onto $B$, $\lambda\in[0,1]$ and $q,(p_n)\subseteq\mathcal{P}(B)$ are such that $p_nqp_n=\lambda p_n$, for all $n$.  Then we can lift $q$ to $Q\in\mathcal{P}(A)$ and $(p_n)$ to $(P_n)\subseteq\mathcal{P}(A)$ such that $P_nQP_n=\lambda P_n$, for all $n$.  For say $Q$ and $P_1,\ldots,P_n$ have been defined and we want to define $P_{n+1}$.  As $P_nQP_n=\lambda P_n$, $P_nQ$ and $P_n^\perp Q$ are well-supported and hence $R=P_n\vee[QP_n]\in A$, and $(p_{n+1}-p_n)\pi(R)=0$ (by the comment after (\ref{split})) so we can find $P\in\mathcal{P}(A)$ with $\pi(P)=p_{n+1}-p_n$ and $PR=0$.  By the proof of \autoref{sigmapq} (and the comment after), we can then adjust $P$ so that in addition we have $PQP=\lambda P$.  Then simply let $P_{n+1}=P_n+P$ and continue the recursion.  From this it follows that if $A$ has the upwards $\lambda$-$\omega$-property  then so does $B$.

Lastly we show that a result proved in \cite{Bice2009} Theorem 6.2 for von Neumann algebras actually holds for all C*-algebras of real rank zero.  The real meaning is perhaps a little lost in its full generality, so for motivational purposes, say we have two closed subspaces $V$ and $W$ of a Hilbert space $H$.  Unless $V+W$ is itself closed, we could not expect to find a maximal closed subspace of $V+W$, as any proper closed subspace of $V+W$ will have a 1-dimensional extension in $V+W$, which is also necessarily closed.  However, it is still possible\footnote{In fact, this will happen precisely when $1$ is not a limit point of $\sigma_e(P_VP_W)$ (even though $V+W$ not being closed means that $1$ is a limit point of $\sigma(P_VP_W)$)} for $V+W$ to have a maximal closed subspace with respect to \emph{essential inclusion}, as given in \cite{Weaver2007} Definition 3.2.  Specifically we say a closed subspace $X$ is essentially included in a closed subspace $Y$ if $\pi(P_X)\leq\pi(P_Y)$, where $\pi$ is the canonical homomorphism to the Calkin algebra.  The surprising thing the following result tells us is that a closed subspace of $V+W$ that is maximal with respect to essential inclusion must in fact be a \emph{maximum} with respect to essential inclusion.

\begin{thm}
Assume $\pi$ is a homomorphism from a C*-algebra $A\subseteq\mathcal{B}(H)$ to $\mathcal{B}(H')$ and take $(P_n)\subseteq\mathcal{P}(A)$ and $P\in\mathcal{P}(A)$ with $\mathcal{R}(P)\subseteq\sum\mathcal{R}(P_n)$.  We have \ref{club1}$\Leftrightarrow$\ref{club2}$\Rightarrow$\ref{club3} where
\begin{enumerate}
\item\label{club1} $\pi(P)=\bigvee\pi(P_n)(=$ the projection onto $\overline{\sum\mathcal{R}(\pi(P_n))})$.
\item\label{club2} $\pi(P)\geq\pi(Q)$ whenever $Q\in\mathcal{P}(A)$ and $\mathcal{R}(Q)\subseteq\sum\mathcal{R}(P_n)$.
\item\label{club3} $\pi(P)=\pi(Q)$ whenever $Q\in\mathcal{P}(A)$ and $\mathcal{R}(P)\subseteq\mathcal{R}(Q)\subseteq\sum\mathcal{R}(P_n)$.
\end{enumerate}
If $A$ has real rank zero then we also have \ref{club3}$\Rightarrow$\ref{club2}, i.e. these statements are all equivalent.
\end{thm}

\begin{proof}
First note that $\mathcal{R}(P)\subseteq\sum\mathcal{R}(P_n)$ implies that we actually have $\mathcal{R}(P)\subseteq\mathcal{R}(P_1)+\ldots+\mathcal{R}(P_m)$ for some $m$.  This is equivalent to saying there exists $\lambda>0$ such that $P\leq\lambda(P_1+\ldots+P_m)$, by \cite{FillmoreWilliams1971} Theorems 2.1 and 2.2.  This, in turn, implies that $\pi(P)\leq\lambda(\pi(P_1)+\ldots+\pi(P_m))$ and hence $\mathcal{R}(\pi(P))\subseteq\mathcal{R}(\pi(P_1))+\ldots+\mathcal{R}(\pi(P_m))\subseteq\sum\mathcal{R}(\pi(P_n))$.  If \ref{club2} holds then, in particular, $\pi(P)\geq\pi(P_n)$, for all $n$, and hence $\sum\mathcal{R}(\pi(P_n))\subseteq\mathcal{R}(\pi(P))$, giving $\pi(P)=\bigvee\pi(P_n)$, i.e. \ref{club2}$\Rightarrow$\ref{club1}.  But the argument above applied to $Q$ instead of $P$ shows that $\mathcal{R}(Q)\subseteq\sum\mathcal{R}(P_n)$ implies $\mathcal{R}(\pi(Q))\subseteq\sum\mathcal{R}(\pi(P_n))$, giving \ref{club1}$\Rightarrow$\ref{club2}.

The \ref{club2}$\Rightarrow$\ref{club3} part is immediate, so assume $A$ has real rank zero and that \ref{club2} fails, i.e. $\pi(P)\ngeq\pi(Q)$ for some $Q\in\mathcal{P}(A)$ with $\mathcal{R}(Q)\subseteq\sum\mathcal{R}(P_n)$.  Picking $\delta\in(0,||\pi(P^\perp Q)||^2/2)$, we have $R\in\mathcal{P}(A)$ such that $E^\perp_{QP^\perp Q}(2\delta)\leq R\leq E^\perp_{QP^\perp Q}(\delta)\leq Q$.  Thus $||PR||\leq\sqrt{1-\delta}$ and hence $P\vee R\in\mathcal{P}(A)$ and $\mathcal{R}(P)\subseteq\mathcal{R}(P\vee R)\subseteq\mathcal{R}(P)+\mathcal{R}(Q)\subseteq\sum\mathcal{R}(P_n)$, even though $||\pi(P^\perp(P\vee R))||\geq||\pi(P^\perp R)||=||\pi(P^\perp Q)||>0$ and hence $\pi(P)\neq\pi(P\vee R)$.
\end{proof}

\bibliography{maths}{}
\bibliographystyle{plainurl}

\end{document}